\documentclass[11pt]{amsart}
\usepackage[T1]{fontenc}
\usepackage[latin1]{inputenc}
\usepackage{a4wide}
\usepackage{setspace}
\usepackage{supertabular}
\usepackage{xypic}
\usepackage{amssymb}
\usepackage{amsthm}
\usepackage[english]{babel}
\usepackage{latexsym}
\date{}
\newtheorem{thm}{Theorem}[section]

\newtheorem{example}[thm]{Example}
\newtheorem{defn}[thm]{Definition}

\newtheorem{prop}[thm]{Proposition}

\newtheorem{lem}[thm]{Lemma}
\newtheorem{cor}[thm]{Corollary}

\newenvironment{f-proof}[1][\sc D\'emonstration.]{\begin{trivlist}
\item[\hskip \labelsep {\bfseries #1}]}{\hfill{$\square$}\end{trivlist}}

\newcommand{\fonc}[5]{
 \begin{array}{cccc}
 #1: & #2 & \longrightarrow & #3\\
     & #4 & \longmapsto & #5
 \end{array}
}

\begin{document}
\title{On the trace form of Galois algebras}

\author[Ph. Cassou-Nogu\`es]{Ph. Cassou-Nogu\`es }
\address{Philippe Cassou-Nogu\`es, IMB\\ Univ. Bordeaux \\  33405 Talence, France.\\}
 
 \email{Philippe.Cassou-Nogues@math.u-bordeaux.fr}

\author[T. Chinburg]{T. Chinburg*}\thanks{*Supported by NSF Grant \#DMS0801030}
\address{Ted Chinburg, Dept. of Math\\Univ. of Penn.\\Phila. PA. 19104, U.S.A.}
\email{ted@math.upenn.edu}

\author[B. Morin]{B. Morin**}\thanks{** Supported by ANR-12-BS01-0002 and ANR-12-JS01-0007}
\address{Baptiste Morin, IMB\\ Univ. Bordeaux \\  33405 Talence, France. \\}

 \email{Baptiste.Morin@math.u-bordeaux.fr}

\author[M. Taylor]{M. J. Taylor}
\address{Martin J. Taylor, Merton College \\ 
Oxford OX1 4JD, U.K.}
\email{martin.taylor@merton.ox.ac.uk}


\maketitle
\begin{abstract}
 We study the trace form $q_L$ of $G$-Galois algebras $L/K$ when $G$ is a finite group and $K$ is a field of characteristic different from $2$. We introduce in this paper the category of $2$-reduced groups and, when $G$ is such a group, we use a formula of Serre to compute the second Hasse-Witt invariant of $q_L$. By combining this computation with work of Quillen we determine the isometry class of $q_L$ for  large families of $G$-Galois algebras over global fields. We also indicate how our results generalize to Galois $G$-covers of schemes. 
\end{abstract}
\section {Introduction}
 We denote by $K$ a field of characteristic different from $2$, by $K^s$ a separable closure of $K$ and by $G_K$ the Galois group of $K^s /K$.  If  $q$ is a  quadratic form of rank $n$,  over a field $K$,  then we may diagonalise $q$ and write $q=<a_1,  \cdots,  a_n>$, for $a_i\in K^{\times}$. 
 
  Let $G$ be a finite group and  let $L/K$ be a $G$-Galois algebra.  We attach to this algebra the so called {\it trace form}.  This is the 
$G$-quadratic form $q_L: L\rightarrow K$ defined by $$q_L(x)=\mathrm{Tr}_{L/K}(x^2).$$ When the degree of $L/K$ is odd, Bayer and Lenstra \cite{Bayer90}  have proved that $L$ has a normal and self-dual basis over $K$;  therefore   $q_L$   is isometric to the unit form    $<1, \cdots, 1>$.  Their  result does not  generalize to the case of algebras of even degree; so for instance   a quadratic extension does not have a self-dual normal basis. In \cite{Bayer94}, Bayer and Serre have given  criteria to ensure the existence of such a self-dual normal basis,  depending on the Sylow $2$-subgroups of $G$. Other   authors have studied the trace form for Galois extensions $L/K$  of even degree either when the degree is small    or when $K$ is a number field (see  \cite{CP}, Theorem I.9.1, \cite{DEK} and \cite{M}).  If  $L/{\bf  Q}$ is a Galois extension of even degree and if the Sylow $2$-subgroups of $\mathrm{Gal}(L/{\bf  Q})$ are non-metacyclic,  then one can prove that either $q_L\simeq <1, \cdots, 1>$ if $L$ is totally real, or that the class of $q_L$ is trivial in the Witt ring of ${\bf Q}$ if $L$ is totally imaginary.  The key-tool in the proof of this result is the Knebusch exact sequence of Witt rings. 

 Another important tool in the classification of quadratic forms is provided by their  Hasse-Witt invariants. They are cohomological invariants $\{w_m(q) \in H^m(G_K, {\bf Z}/2{\bf Z}), m \geq 0\}$ in the cohomology mod $2$ of the profinite  group $G_K$.   
In this paper we  study the trace forms of $G$-Galois algebras   of even degree, over any arbitrary field of characteristic different from $2$,  by computing their Hasse-Witt invariants at least in small degrees. As we will see later these invariants are  related to  classes in  the $\mathrm{mod}\ 2$ cohomology ring of $G$.  The computation of the cohomology ring of finite groups appears in a myriad of contexts. It plays an important role in the work of Quillen ( \cite {Q71}, \cite{Q72} and \cite{Qu72}).  We will make use  of several of his results in this paper.  We introduce  the following definition:

\begin{defn} A finite group $G$ is said to be  $2$-reduced if $H^2(G, {\bf Z}/2{\bf Z})$ contains no non-zero nilpotent element of the  $\mathrm{mod}\ 2$ cohomology ring of $G$. 

\end{defn} We observe that various natural families of groups are $2$-reduced. More precisely,   denoting by ${\bf F}_r$ the finite field of $r$ elements,  we obtain:
\begin{thm} \label{2red} The  following groups are $2$-reduced: 
\begin{enumerate}
\item [i)] groups with   Sylow $2$-subgroups which are  either cyclic or   abelian  elementary;
\item [ii)] symmetric groups $S_n$ and  alternating groups $A_n$; 
\item [iii)] dihedral groups; 
\item[iv)]  linear groups ${\bf {Gl}}_n({\bf F}_r), r\equiv 3\ \mathrm{mod}\ 4$;
\item[v)] orthogonal groups ${\bf O}_n({\bf F}_r), r\equiv 1\ \mathrm{mod}\ 4$;
\item [vi)] the Mathieu group  $M_{12}$. 
\end{enumerate}
\end{thm}
\noindent{\bf Remarks.}
{ \bf1)} One should note that for most of these groups one knows that $H^2(G, {\bf Z}/2{\bf Z})\neq 0$. This is the case when $G=A_n, S_n,  D_{2^n}$ and  $M_{12}$.   
\vskip 0.1 truecm
\noindent {\bf 2)} For the sake of simplicity let us call a finite group  {\it reduced}  if its  mod $2$  cohomology ring is reduced. If the groups  $G_1$ and $G_2$ are reduced, then it follows from the K\"unneth formula that the same holds for  $G_1\times G_2$. This is  the case for instance for $G_1=({\bf Z}/2{\bf Z})^n$ and $G_2=D_{2^m}$. Therefore any product of reduced groups  provides us with new families of reduced and so $2$-reduced groups.   Nevertheless one should note that there exist $2$-reduced groups which are not reduced;  every cyclic $2$-group of order greater than $4$ has this property. We now consider    $G={\bf Z}/4{\bf Z}\times {\bf Z}/2{\bf Z}$.  This is a product of $2$-reduced groups,  however   one can prove that $H^*(G, {\bf Z}/2{\bf Z})={\bf F}_2[z, y, x]/(z^2)$ with $z, y$ having degree $1$ and $x$ having degree $2$  (see \cite{CTVZ}, Appendix A) and so one can check that    $H^2(G, {\bf Z}/2{\bf Z})$ contains non-zero nilpotents elements. We  conclude that the product of $2$-reduced groups is not in general a $2$-reduced group. 
\vskip 0.1 truecm
\noindent {\bf 3)}  One can also use the wreath product of groups  for constructing large families of $2$-reduced groups (see Remarks, Section 3.3)
\vskip 0.1 truecm

We now explain how such cohomological considerations relate to the Hasse-Witt invariants of trace forms: indeed this was very much the driving motivation for our results on  the $\mathrm{mod}\ 2$ cohomology ring. So suppose now that $L/K$ is a $G$-Galois algebra,  defined by a group homomorphism $\Phi_L: G_K\rightarrow G$ and let $q_L$ be its trace form.  Serre's comparison formula ((\cite{Serre84}, Theorem 1) provides us with the equality : \begin{equation}\label{S}
w_2(q_L)=\Phi_L^*(c_G)+(2)\cdot(d_{L/K})
\end{equation}
where $d_{L/K}$ is the discriminant of the $K$-algebra $L$ and $\Phi_L^*(c_G)$ is the inverse image by $\Phi_L$ of $c_G\in H^2(G, {\bf Z}/2{\bf Z})$ defined by the group extension 
 $$1\rightarrow {\bf Z}/2{\bf Z} \rightarrow {\bf Pin}(G) \rightarrow G\rightarrow  1, $$
(see (\ref{pin}) Section 4.2  for a precise definition of this extension). We shall prove,  under certain  assumptions on the order of $G$,  that when $G$ is $2$-reduced then this group extension is split. Therefore as a consequence of this result and the equality (\ref {S}) we will obtain:

\begin{thm}\label{main} Let $G$ be a $2$-reduced  group of order $n$, $n\equiv  0\ \mathrm{or}\ 2\ \mathrm{mod}\ 8$.  Then for any $G$-Galois algebra $L/K$ one has: 
$$w_2(q_L)=(2)\cdot (d_{L/K}).$$
 
\end{thm}

\begin{cor}\label{triv}  Let $G$ be a $2$-reduced  group of order $n, n\equiv 0\  \mathrm{mod}\ 8$. We assume that the Sylow $2$-subgroups of $G$ are non-cyclic. Then for any $G$-Galois algebra $L/K$ one has: 
$$w_1(q_L)=w_2(q_L)=0.$$ 
\end{cor} 

We note that Corollary \ref{triv} can be slightly generalized in the following way: 
\begin{cor}\label{trivd} Let $G$ be a group of order $n, n\equiv 0\  \mathrm{mod}\ 8$ and let $S$ be a Sylow $2$-subgroup of $G$. Suppose that: 
\begin{enumerate}
\item [i)] $S$ is non-cyclic;
\item [ii)] $S$ is the Sylow $2$-subgroup of some  $2$-reduced  group $H$. 
\end{enumerate}
Then for any $G$-Galois algebra $L/K$ one has 
$$w_1(q_L)=w_2(q_L)=0.$$
\end{cor} 
\noindent {\bf Remark.}  Corollary \ref{trivd}  can be useful in cases where $G$ itself is not $2$-reduced. 
Let  $G=S$ be  the quaternion group of order $8$. We note from the description of the cohomology ring mod $2$ of $G$ (\cite{CTVZ}, Appendix B) that $G$ is not $2$-reduced. 
However, since $G$ can be seen as the Sylow $2$-subgroup of the symmetric group $S_4$, which is  a $2$-reduced group,  we can apply  Corollary \ref{trivd}. We conclude that   if the Sylow $2$-subgroups of a group $G$ are quaternion groups of order $8$,  then, for any $G$-Galois algebra $L/K$ one has $w_1(q_L)=w_2(q_L)=0$.
\newline

If we  now take the  field $K$ to be a {\it  global field}, then  we can use the Hasse-Minkowski Theorem to deduce  from Theorem \ref {main} a precise description of the trace form. 
\begin{cor}\label{ff} Let $K$ be a global function field of characteristic different from $2$ and let $G$ be a  $2$-reduced group of order $n$, $n\equiv  0\ \mathrm{or}\ 2\ \mathrm{mod}\ 8$. Then  for   any $G$-Galois algebra over $K$  one has the following properties:
\begin{enumerate}
\item [i)]$q_L\simeq <1, \cdots,  1>$ if the Sylow $2$-subgroups of $G$ are non-cyclic; 
\item [ii)] $q_L\simeq <2, 2d_{L/K}, 1, \cdots, 1>$ otherwise.
\end{enumerate}
\end{cor} 
Suppose now that $K$ is a number field. For any infinite place $v$ of $K$ we  consider $L_v=L\otimes_{K}K_v$. This is a $G$-Galois algebra on $K_v$. If $v$ is real, 
since $\mathrm{Gal}({\bf C}/{\bf  R})$ is of order $2$, then we can associate to $L_v/K_v$ an element of order 2 of $G$,  which is unique up to conjugacy (see Section 2.1), and that  we denote by $\sigma(L_v)$.  
\begin{cor}\label{numb1} Let $K$ be a number field and let $G$ be a $2$-reduced group of order  $n$, $n\equiv 0\  \mathrm{mod}\ 8$. We assume that the Sylow $2$-subgroups of $G$ are non-cyclic. Then for any $G$-Galois algebra $L/K$ the following properties are equivalent:
\begin{enumerate}
\item [i)] The trace form $q_L$ is isometric to the unit form $<1, \cdots, 1>$;
\item [ii)] $\sigma(L_v)=1$ for any real place $v$ of $K$. 
\end{enumerate}
\end{cor}
\begin{cor}\label{im} Let $G$ be a $2$-reduced group of order $n$,  $n\equiv 0\  \mathrm{mod}\ 8$. We assume that the Sylow $2$-subgroups of $G$ are non-cyclic. Then the trace form of any $G$-Galois algebra over a totally imaginary number field is isomorphic to the unit form $<1, \cdots,  1>$. 
\end{cor}

\noindent {\bf Remark.}  Clearly if $G$ is of odd order, then obviously $H^i(G, {\bf Z}/2{\bf Z})=\{1\}$ for all positive $i$; this is the situation considered in \cite {Bayer90}. This leads us to consider the situation in Corollary \ref{numb1} with  the stronger hypotheses 
$$H^1(G, {\bf Z}/2{\bf Z})=H^2(G, {\bf Z}/2{\bf Z})=0. $$
In this case, then ii) is also equivalent to  $L$ having a self-dual normal basis (\cite{Bayer94} Theorem 3.2.1.). We note that under  our weaker hypotheses we may obtain Galois algebras which have a self-dual basis but do not have a self-dual normal basis. This is in particular the case for any $G$-Galois algebra over an imaginary quadratic field when $G=S_n, n\geq 4$. 
\newline

  Let $L/{\bf Q}$ be a Galois algebra of rank $n$ and let $v_{\infty}$ be  the archimedean place of ${\bf Q}$. If $\sigma(L_{v_{\infty}})=1$ then $L_{v_{\infty}}/{\bf R}$ is split and  so $L/{\bf Q}$ is {\it  totally real};   if $\sigma(L_{v_{\infty}}) \neq 1$, then  $L_{v_{\infty}}/{\bf R}$ is the product of $n/2$ copies of ${\bf  C}$ and  so $L/{\bf Q}$ is {\it totally imaginary}. We set $d_L:=d_{L/\bf  Q}$. If $q$ and $q'$ are quadratic forms then we denote  their direct orthogonal sum  by $q\oplus q'$ and    the direct orthogonal sum of $s$ copies of $q$   by $s\otimes q$. 

\begin{cor}\label{numb2} Let   $G$ be a $2$-reduced group of order  $n$, $n\equiv  0\ \mathrm{or}\ 2\ \mathrm{mod}\ 8$ and let $S$ be a Sylow $2$-subgroup of $G$. Then for any $G$-Galois algebra $L/{\bf  Q}$ we have:  

\begin{enumerate}

\item [i)] $q_L\simeq <1, \cdots, 1>$ if $L$ is totally real and $S$ is non-cyclic; 
\item [ii)] $q_L\simeq \frac{n}{2}\otimes <1, -1>$ and $w_i(q_L)=\binom{\frac{n}{2}}{i}, i \geq 3$ if $L$ is totally imaginary and $S$ is non-cyclic;  
\item [iii)] $q_L\simeq <2, 2d_L, 1, \cdots, 1>$  if $L$ is totally real and $S$ is  cyclic; 
\item [iv)] $q_L\simeq (\frac{n}{2}-1)\otimes<1, -1>\oplus <(-1)^{(\frac{n}{2}-1)}2, 2d_L>$, if $L$ is totally imaginary and $S$ is cyclic. 

\end{enumerate}
\end{cor}
The computation of the Hasse-Witt invariants of $q_L$ in ii) follows immediately from the description of $q_L$ and the observation that for $i\geq 3$ the cup product of $i$-times the class of $(-1)\in H^1(G_{{\bf Q}}, {\bf Z}/2{\bf Z})$ is the non trivial class of $H^i(G_{{\bf Q}}, {\bf Z}/2{\bf Z})\simeq {\bf Z}/2{\bf Z}$. In particular it follows from the equality   $w_i(q_L)=\binom{\frac{n}{2}}{i}$  that $w_i(q_L)=0$ for $i\geq 3$ and odd.   
The triviality of the Hasse-Witt invariants for $i$ odd can also be deduced from the triviality of $w_1(q_L)$, which is true since $S$ is non-cyclic (see Proposition \ref{ww_1}),  and the equality  $w_1(q)\cdot w_{i-1}(q)=w_i(q)$  for any Galois algebra $L/K$ and any odd integer $i$ (see \cite{GMS}, (19.3)).

\begin {example} 1) 
 The  splitting  field of the polynomial $X^4-X^3-4X-1$ is   a totally real Galois extension of ${\bf  Q}$  with Galois group $S_4$;   hence  its trace form is isometric to the unit form.

2)  The splitting field of $X^4-2X^2-4X-1$ is  a totally imaginary Galois extension of ${\bf  Q}$, with Galois group $S_4$; hence its trace form is isometric  to  $ 12\otimes<1, -1>$.
\end{example}
To  complete the study of the trace form we add in  Section 5 a brief proof of a slight generalization of Conner and Perlis result (\cite{CP}, Theorem I.9.1).
\begin{prop}\label{cp} Let $K$ be a global field and let $L/K$ be a $G$-Galois algebra. Assume that the Sylow $2$-subgroups of $G$ are non-metacyclic. Then 
\begin {enumerate}
\item [i)]If $K$ is a function field of characteristic different from $2$ then the trace form is isometric to the unit form.
\item [ii)]If $K={\bf  Q}$ , the following assertions are equivalent: 
\begin{enumerate}
\item [a)] The trace form $q_L$ is isometric to the unit form $<1, \cdots, 1>$;
\item [b)] $L$ is totally real.
\end{enumerate}
\end{enumerate}
\end{prop}

We now describe the structure and the  content of the paper. In Section 2 we recall some basic properties of Galois algebras and Hasse-Witt invariants of quadratic forms. Section 3 is dedicated to the study of $2$-reduced groups and contains the proof of Theorem \ref{2red}. In Section 4 we compute the Hasse-Witt invariants of  degree $1$ and $2$ of the trace form of $G$-Galois algebras when the group $G$ is $2$-reduced; we  prove Theorem  \ref{main} and some of its corollaries. In Section 5 we assume that the base field $K$ is a global field and we prove some corollaries of Theorem \ref{main} in this case. Finally, in the last section,  we show how our results apply to a geometric set-up where we replace Galois algebras by Galois covers of schemes. 
\section{Preliminaries}

We recall that in this paper $K$ is a field of characteristic different from $2$, $K^s$ is a separable closure of $K$ and $G_K$ is the Galois group $\mathrm{Gal}(K^s/K)$. 
\subsection {Galois algebras} Let $G$ be a finite group. A $G$-Galois algebra over $K$  is an etale $K$-algebra $L$  of degree $n=|G|$, endowed with an action of $G$  such that the action of $G$ on 
$X(L)=\mathrm{Hom}^{\mathrm{alg}}(L, K^s)$  is simply transitive. The group $G_K$ acts by composition on $X(L)$. Fixing an element $\chi \in X(L)$ we attach to $L$ a group homomorphism $\Phi_L: G_K\rightarrow G$ defined by 
\begin{equation}\label{action1}\omega\chi=\chi\Phi_L(\omega)\ \forall \omega \in G_K.\end{equation}
We note that $\Phi_L$ is independant of the choice of $\chi$ up to conjugacy. If we denote by $E$ the subfield $\chi(L)$ of $K^s$, then $E$ is a Galois extension of $K$, with Galois group $\mathrm{Im}(\Phi_L)$,  and the algebra $L$ is $K$-isomorphic to the product of $m$ copies of $E$ where $m$ is the index of $\mathrm{Im}(\Phi_L)$ in $G$. This implies an isometry  
\begin{equation}\label{whit} q_L\simeq m\otimes q_E \end{equation}
of quadratic forms. Indeed when $\Phi_L$ is surjective the $G$-algebra $L$ is a Galois extension of $K$ with Galois group  $G$. In the case where $K={\bf R}$, the group $G_K$ is of order $2$ and so  $\Phi_L$ is defined,  up to conjugacy,  by an element $\sigma(L)\in G$ such that $\sigma(L)^2=1$. 

 We denote by $S(G)$ the group of permutations of $G$ and by $f: G\rightarrow S(G)$ the group  homomorphism induced by the action of $G$ on itself by left multiplication.  We may identify $G$ and $X(L)$ as sets via the map $g\rightarrow \chi g$. Under this identification the action of $G_K$ on $X(L)$ provides us with a group homomorphism 
\begin{equation}\label{action2}
\fonc{\varphi_L}{G_K}{S(G)}
{\omega}{(g\rightarrow \Phi_L(\omega)g)}
\end{equation} which is the composition of $\Phi_L$ with $f$. 
Identifying $G$ with $[1, \cdots, n]$ then  $f$ and $\varphi_L$ become respectively  group homomorphisms $f: G\rightarrow S_n$ and  $\varphi_L: G_K\rightarrow S_n$.  

\subsection {Hasse-Witt invariants}
 If   $q$ is a  non-degenerate quadratic form of rank $n$ over $K$, we   choose a diagonal form $<a_1,\cdots,a_n> $ of $q$ with $a_i\in K^{\times}$, and consider the cohomology classes $$(a_i)\in K^{\times}/(K^{\times })^2\simeq H^1(G_{K},\mathbb{Z}/2\mathbb{Z}).$$ 
For $1\leq m\leq n$, the \emph{$m$-th  Hasse-Witt invariant} of $q$ is defined to be  
\begin{equation}\label{inv1} w_m(q)=\sum_{1\leq i_1<\cdots<i_m\leq n} (a_{i_1})\cdots(a_{i_m})\in H^m(G_{K},\mathbb{Z}/2\mathbb{Z})\end{equation}
where $(a_{i_1})\cdots(a_{i_m})$ is the cup product. Furthermore we set $w_0(q)=1$ and $w_m(q)=0$ for $m>n$. It can be shown that $w_m(q)$ does not depend on the choice of the particular diagonalisation  of $q$.

In the case where $L/K$ is a $G$-Galois algebra  as considered in Section 2.1, it follows from the Whitney formula for the Hasse-Witt invariants of quadratic forms that  (\ref{whit}) implies the equalities: 
\begin{equation}\label{rel}
w_1(q_L)=mw_1(q_E)\ \mathrm{and}\ w_2(q_L)=\binom{m}{2} w_1(q_E)\cdot w_1(q_E)+mw_2(q_E).
\end{equation} 
\section { 2-reduced groups} 
\subsection{The $2$-lift  property}
For a finite group  $G$ we consider the group extensions of $G$ by ${\bf Z}/2{\bf Z}$:
$$1\rightarrow {\bf Z}/2{\bf Z}\rightarrow G'\rightarrow G\rightarrow 1. $$  The isomorphism classes of such extensions correspond bijectively to the group 
$H^2(G,  {\bf Z}/2{\bf Z})$. An extension is {\it split} if it corresponds to the zero class of $H^2(G,  {\bf Z}/2{\bf Z})$. In that case $G'$ is isomorphic to the direct product ${\bf Z}/2{\bf Z}\times G$. 

 For a subgroup $H$ of $G$ we let  $res^G_H$ denote the restriction map $$H^2(G,  {\bf Z}/2{\bf Z})\rightarrow H^2(H, {\bf Z}/2{\bf Z}).$$  Let  $\mathcal{S}$ be  the set of subgroups of $G$ of order $2$.  We consider the group homomorphism 
\begin{equation}\label{res}
\fonc{s_G}{H^2(G, {\bf Z}/2{\bf Z})}{\prod_{T\in \mathcal{S}}H^2(T, {\bf Z}/2{\bf Z})}
{x}{ (res^G_T(x))_{T\in \mathcal{S}}}
.\end{equation}
\begin{defn} An extension of $G$ by ${\bf Z}/2{\bf Z}$ is said to have  the $2$-lift property if it defines an element of $\mathrm{Ker}(s_G)$. Similarly   an element of $H^2(G,  {\bf Z}/2{\bf Z})$ is said to have   the $2$-lift property if it belongs to  $\mathrm{Ker}(s_G)$. 
\end{defn}
We note that the terminology is justified by the following tautological lemma:
\begin{lem} The following assumptions are equivalent:
\begin{enumerate}
\item $1\rightarrow {\bf Z}/2{\bf Z}\rightarrow G'\rightarrow G\rightarrow 1$ has the $2$-lift  property;
\item every element of $G$ of order $2$ has a lift in $G'$ of order $2$.
\end{enumerate}
\end{lem}
\noindent {\bf Remark.} It follows from the properties of the restriction map that for any subgroup $H$ of $G$ we have the inclusion: 
\begin{equation}\label{inc}res_H^G(\mathrm{Ker}(s_G))\subset \mathrm{Ker}(s_H).\end{equation}

\subsection{A cohomological characterization}
In this section   we shall be particularly  interested  by the groups $G$ such that $\mathrm{Ker}(s_G)=0$, namely  the groups $G$ such that the split extension is the unique extension of $G$ by ${\bf Z}/2{\bf Z}$ which has the $2$-lift property. We recall that a finite group $G$ is said to be {\it $2$-reduced group}  if   $H^2(G, {\bf Z}/2{\bf Z})$ contains no non-zero nilpotent of 
$H^*(G, {\bf Z}/2{\bf Z})$. 

\begin{thm}\label{nil} Let $G$ be a finite group.  Then the following assumptions are equivalent:
\begin{enumerate}
\item $\mathrm{Ker}(s_G)=0$;
\item the group $G$ is $2$-reduced. 
\end{enumerate}
\end{thm}
\begin{proof}  
  The proof of the theorem is an immediate consequence of the following lemma: 
 \begin{lem}\label{key}Let $x$ be an element of $H^2(G, {\bf Z}/2{\bf Z})$. Then the following properties are equivalent: 
 \begin{enumerate}
  \item $x$ is a nilpotent element of the cohomological ring $H^*(G, {\bf Z}/2 {\bf Z})$; 
 \item $x$ has the $2$-lift property.
 \end{enumerate}
 \end{lem} 
  \begin{proof} Let $x\in H^2(G, {\bf Z}/2{\bf Z})$ be a nilpotent element of $H^*(G, {\bf Z}/2{\bf Z})$. For $T\in \mathcal{S}$, the even degree subring $H^{2*}(T, {\bf Z}/2{\bf Z})$  of 
$H^{*}(T, {\bf Z}/2{\bf Z})$ is isomorphic to the  polynomial ring $\mathbb F_2[z_2]$ in one variable,  generated by the generator  $z_2$ of $H^2(T, {\bf Z}/2{\bf Z})$. Since  this ring  is reduced,  we conclude that $res^G_T(x)=0$   and so that $x$, by definition, has the $2$-lift  property.  We now consider an element $x\in H^2(G, {\bf Z}/2{\bf Z})$ having the $2$-lift property. It follows from (\ref{inc}) that for any abelian $2$-elementary  subgroup $H$, then   $res^G_H(x)$ has the $2$-lift   property. We now have: 
\begin{lem}\label{split} For any   elementary abelian  $2$-group $H$, then we have $\mathrm{Ker}(s_H)=0$.
\end{lem}
\begin{proof}   Suppose that  $1\rightarrow {\bf Z}/2{\bf Z}\rightarrow H'\rightarrow H\rightarrow 1$ is an exact sequence having the $2$-lift  property. Then any $h$ of $H'$ is   the lift of an element of $H$  and so satifies $h^2=1$. Therefore $H'$ is an abelian $2$-elementary group and the sequence is split. 
\end{proof}
 It follows from   Lemma \ref{split}  that $res^G_H(x)=0$ for any abelian $2$-elementary subgroup.  By  a theorem of Quillen  \cite{Q} we know   that  every element  $x\in  H^*(G, {\bf Z}/2{\bf Z})$ which restricts to zero on any elementary abelian $2$-subgroup of $G$  is nilpotent. Therefore we conclude that $x$ is nilpotent. This completes the proof of Lemma \ref{key}. 
\end{proof}
\end{proof}

\noindent{\bf Remark.} We note that if $G$ is the abelian elementary group $({\bf Z}/2{\bf Z})^n$ the cohomological ring $H^*(G, {\bf Z}/2{\bf Z})$ is a  polynomial ring   
$\mathbb F_2[x_1, ...,x_n]$ and then, as expected,  has no non zero nilpotent element. 

It is useful to note the following result:
\begin{cor}\label{sylow} Let $G$ be a finite group. Suppose that  the Sylow $2$-subgroups of $G$ are $2$-reduced  then $G$ is $2$-reduced. 
\end{cor}
\begin {proof}
Let $S$ be a Sylow $2$-subgroup of $G$.  The group $S$ being $2$-reduced,  it follows from (\ref{inc}) that 
$$res_{S}^G(\mathrm{Ker}(s_G))\subset \mathrm{Ker}(s_{S})=0.$$  Since the index of $S$ in $G$ is odd,  the restriction map  is an injection and  so  
$\mathrm{Ker}(s_G))=0$.
\end{proof}

\subsection{Proof of Theorem \ref{2red}}
Our aim is to check that every group appearing in Theorem \ref{2red} is  $2$-reduced. It follows from Corollary \ref{sylow} that in order to prove i) it suffices to prove that cyclic or abelian elementary $2$-groups are $2$-reduced. The case of abelian elementary $2$-groups  has been treated in Lemma \ref{split}.  Let 
\begin{equation}\label{splitsplit}1\rightarrow {\bf Z}/2{\bf Z}\stackrel{i}\rightarrow G'\stackrel{s}\rightarrow G\rightarrow 1\end{equation} be an extension with the $2$-lift  property. We set   $\mathrm{im}(i)=T=\{e, t\}$. 
\begin{lem}\label{cyc} Assume that $G$ is a $2$-group. Then for any cyclic subgroup $V$ of $G$  the subgroup $s^{-1}(V)$ is  abelian and  equal to a direct product of  $T$ by a subgroup $U$ of $G'$.  
\end{lem}
\begin{proof} Since  $T$ is a central subgroup of $s^{-1}(V)$ such that $s^{-1}(V)/T$ is cyclic then $s^{-1}(V)$ is an abelian group. Take a generator  $v$ of  $V$ and take $U$ as the subgroup generated by a lift $u$ of $v$.  Since the extension has the $2$-lift  property then $U$ is a cyclic group of order equal to the order of $V$ which does not contain $t$. We conclude that  $s^{-1}(V)$ is the direct product of the subgroup $U$ and $T$.
\end{proof}
When $G$ is a cyclic $2$-group we may  use Lemma \ref{cyc} with   $V=G$ and  conclude that every extension of $G$ with the $2$-lift  property is split. 

In order to study extensions of $G$ having the $2$-lift  property, Theorem \ref{nil} leads us to study more precisely the cohomology algebra $H^*(G, {\bf Z}/2{\bf Z})$.  Following Quillen \cite{Q71} we shall say that a family $\{H_i\}_{i\in I}$ of subgroups of  $G$  is {\it a detecting family}, if the map
$$H^*(G, {\bf Z}/2{\bf Z})\rightarrow \prod_{i\in I}H^*(H_i, {\bf Z}/2{\bf Z})$$ given by the restriction homomorphisms is injective. Since the $2$-lift  property is stable under any restriction map we deduce that any group which has a detecting family of $2$-reduced  subgroups is $2$-reduced. This is precisely the case for symmetric, dihedral, linear groups ${\bf {Gl}}_n( {\bf F}_r))$,  orthogonal groups ${\bf O}_n({\bf F}_r)$,  and $M_{12}$ where  the family of  elementary abelian $2$-subgroups provides us with a family of detecting groups (see \cite{Q71}, Corollary 3.5, Theorem 4.3 (4-5) and Lemma 4.6, \cite{Qu72}, Lemma 13 and \cite{AM}, VIII, Section 3.)  which, according to  Lemma \ref{split},  are $2$-reduced. 

Suppose  now that $G$ is the alternating group $A_n$, $n\geq 4$.  We know (\cite {Serre84}, Section 1.5) that the unique non trivial class of $H^2(A_n, {\bf Z}/2{\bf Z})$ is the restriction $res^{S_n}_{A_n}(s_n)$ where $s_n\in H^2(S_n, {\bf Z}/2{\bf Z})$ corresponds  to the extension 
\begin{equation}\label{sn}1\rightarrow {\bf Z}/2{\bf Z}\rightarrow \tilde {S_n} \rightarrow S_n\rightarrow 1\end{equation} which is 
   characterized by the property  that   transpositions in $S_n$ lift to elements of
order $2$, while products of two disjoint transpositions lift to elements 
of order $4$. We conclude that $res^{S_n}_{A_n}(s_n)$ does not have the $2$-lift property since a product of two disjoint transpositions has a lift of order $4$. Hence  $A_n$ is $2$-reduced. This completes the proof of the theorem.  \hfill $\square$

\noindent{\bf Remarks.} {\bf 1})    We can also deduce that $S_n$ is a $2$-reduced group from the description of      $H^2(S_n, {\bf Z}/2{\bf Z})$ given in \cite{Sc}.  This group   is a non-cyclic group of order $4$ for $n\geq 4$. The first of the three non-trivial extensions is the extension 

  $$1\rightarrow {\bf Z}/2{\bf Z}\rightarrow \tilde{S}_n \rightarrow S_n \rightarrow 1  $$ given in  (\ref{sn}) above. The second such extension is the extension 

  $$1 \rightarrow {\bf Z}/2{\bf Z} \rightarrow S'_n \rightarrow S_n\rightarrow 1$$ which is obtained by pulling  back, via  the sign character $\varepsilon_n:S_n\rightarrow {\bf C}^{\times}$,   the Kummer sequence

\begin{equation}\label{Kum}  1\rightarrow {\bf Z}/2{\bf Z}\simeq {\pm 1} \rightarrow {\bf C}^{\times}\rightarrow  {\bf C}^{\times}\rightarrow 1,  \end{equation}
induced by squaring on ${\bf C}^{\times}$.     
We prove  that in this case {\it  the lift in $S'_n$ of
any transposition in $S_n$ has order 4}. The third and final such extension is the extension 

  $$1\rightarrow  {\bf Z}/2{\bf Z}\rightarrow  S''_n \rightarrow  S_n \rightarrow  1$$
which  represents the class of the  sum of the two previous ones  in $H^2(S_n, {\bf Z}/2{\bf Z})$.  By the definition of Baer sums,
we may describe  $S''_n$ and prove that  {\it any lift in $S''_n$ of a transposition in $S_n$ has order $4$.} 
Therefore we conclude  that the unique extension of $S_n$ by ${\bf Z}/2{\bf Z}$ having the $2$-lift  property is the split extension and so that $S_n$ is $2$-reduced.
\vskip 0.1 truecm
\noindent {\bf 2)} Let $G$ be a group and let  $G\int {\bf Z}/2{\bf Z}$ be the wreath product. Recall that $G\int {\bf Z}/2{\bf Z}$ is the semi-direct product $G^2\rtimes {\bf Z}/2{\bf Z}$ where ${\bf Z}/2{\bf Z}$ is identified with the symmetric group $S_2$ and acts on $G^2$ by permuting the factors. Suppose that the set of elementary abelian $2$-subgroups  is a detecting family for the group $G$. Then it follows from a theorem of Quillen (see \cite{AM}, Theorem 4.3) that the same property holds for the wreath product $G\int {\bf Z}/2{\bf Z}$. Therefore we conclude that every group $G\int {\bf Z}/2{\bf Z}\int ...\int {\bf Z}/2{\bf Z}$ is  $2$-reduced. 
\vskip 0.1 truecm
 \noindent {\bf 3}) We know from Theorem \ref {2red} that amongst the groups of order $8$ the cyclic group, the elementary abelian $2$-group and the dihedral group are $2$-reduced.  One should note  that  on the contrary the quaternion group  and the abelian group ${\bf Z}/4{\bf Z}\times {\bf Z}/2{\bf Z}$ are not. Let us treat as an example the case of the quaternion group.  Let $G'$ be the semi-direct product of two cyclic groups of order $4$ defined by the presentation 
$$<u_1, u_2\ | u_1^4=u_2^4=e, u_2u_1u_2^{-1}=u_1^{-1}>.$$
One notes  that $u_1^2, u_2^2$ and $u_1^2u_2^2$ are the elements of order $2$ of $G'$ and that  $Z(G')=\{e, u_1^2, u_2^2, u_1^2u_2^2\}$ is the center of $G'$.  We set $T=\{e, u_1^2u_2^2\}$ and $G=G'/T$ and  we consider the exact sequence 
\begin{equation} \label{quat}1\rightarrow T\rightarrow G' \rightarrow G\rightarrow 1. \end{equation}
The group $G$ is the quaternion group of order $8$ and the extension (\ref {quat}) has the $2$-lift  property. One checks that every subgroup $H$  of $G'$ of order $8$ contains at least $2$ distinct elements  of order $2$. Therefore $H$ contains $Z(G')$ and  is commutative since $H/Z(G')$ is cyclic.  We conclude that $G'$ does not contain any quaternion subgroup of order $8$ and so that (\ref {quat}) is not split. 

\section{Hasse-Witt invariants of the trace form}

In this section we consider a $G$-Galois algebra where $G$ is a finite group and we denote  its trace form by $q_L$. We attach to $L/K$ the group homomorphisms  $\Phi_L: G_K\rightarrow G$ and $\varphi_L: G_K\rightarrow S_n$ introduced in Section 2. We recall that $\varphi_L$ is the composition of $\Phi_L$ with the group homomorphism $f:G\rightarrow S_n$ induced by left multiplication of $G$ on itself. Our aim is to compute the Hasse-Witt invariants of the trace form $q_L$. 
\subsection{The invariant  $w_1(q_L)$}
The  discriminant of the form $q_L$ is   by definition   the discriminant $d_{L/K}$ of the etale algebra $L/K$. The Hasse-Witt invariant $w_1(q_L)$ is the class $(d_{L/K})$ defined by this discriminant in $H^1(G_K, {\bf Z}/2{\bf Z})$. As a group homomorphism $G_K\rightarrow {\bf Z}/2{\bf Z}$ it  is the composition  $\varepsilon_n\circ \varphi_L$ where 
$\varepsilon_n:S_n\rightarrow \{\pm 1\}\simeq {\bf Z}/2{\bf Z}$ is the signature map. Thus,  $w_1(q_L)=0$ if and only if $f(\mathrm{Im}(\Phi_L))\subset A_n$. Indeed this will be always  the case if the order of $G$ is odd. We now consider the case where the rank of $L/K$ is even. The following proposition is well known at least for Galois extensions (see 
\cite{CP} Theorem 1.3.4.)

\begin{prop}\label{ww_1} $L/K$ be a $G$-Galois algebra of finite even degree. Then $w_1(q_L)=0$ if and only if one of the following assumptions is satisfied:
\begin{enumerate}
\item the Sylow $2$-subgroups of $G$ are non-cyclic;
\item the index of $\mathrm{Im}(\Phi_L)$ in $G$ is even. 
\end{enumerate}
\end{prop}
\begin{proof} We start by proving a lemma. 
\begin{lem}\label {et} Let $G$ be a finite group of even order $n$ then the following properties are equivalent:
\begin {enumerate}
\item $\mathrm{Im}(f)\subset A_n$; 
\item the Sylow $2$-subgroups of $G$ are non-cyclic. 
\end{enumerate}
\end{lem}
\begin{proof} We write $n=2^an' $ with $a\geq 1$ and $n'$ odd.   Let $g\in G$ be an element of   $2$-power order,  $2^b$ say,   $b\leq a$. Each orbit of $\tilde g:=f(g)$ acting on $[1,...,n]$ is of length $2^b$ and so $\tilde g$ decomposes into a product of $2^{a-b}n'$ disjoint cycles of length $2^b$. Therefore we deduce that 
$$\varepsilon_n(\tilde g)=(-1)^{(2^b-1)2^{a-b}n'}=(-1)^{(n-2^{a-b}n')}.$$
We conclude that if the $2$-Sylow subgroups of $G$ are not cyclic the image by $f$ of any $2$-power order element of $G$ belongs to $A_n$ and so that $\mathrm{Imf}\subset A_n$, whereas,  if the $2$-Sylow subgroups of $G$ are cyclic, then the signature of  the image by $f$ of any element of order $2^a$ is  odd. 
\end{proof}
Following  the proof of the Lemma we observe that $w_1(q_L)=0$ if and only if $\mathrm{Im}(\Phi_L)$ does not  contain any element of order $2^a$ that is to say  if and only if (1) or (2) is satisfied. 

\end{proof}
\begin{cor} Let $L/K$ be a $G$-Galois algebra of either odd degree or of even degree,  satisfying the assumptions  of  Proposition \ref{ww_1}; then $w_i(q_L)=0$ if $i$ is odd. 
\end{cor}
\begin{proof} The result is an immediate consequence of Proposition \ref {ww_1} since we know that for any non-degenerate quadratic form and any odd integer $i$ the following equality holds:
$$w_1(q)\cdot w_{i-1}(q)=w_i(q)$$ (see \cite{GMS}, (19.3)).
\end{proof}

\subsection {The group $\mathrm{\bf Pin}(G)$}

  Let  $(V, q)$ be a quadratic form over $K$. We denote the  Clifford algebra of $q$ by $Cl(q)$. Recall that this is the quotient algebra $T(V)/J(q)$,  here $T(V)$ is the tensor algebra of $V$ and $J(q)$ is the two-sided ideal of $T(V)$ generated by the elements $x\otimes x-q(x)1$ when $x$ runs through the elements of $V$. We shall view $V$ as embedded in $Cl(q)$ in the natural way. If we write $q=<a_1, \cdots, a_n>$ with  orthogonal basis $\{e_1, \cdots,  e_n\}$, then $Cl(q)$ is generated as an algebra by the $e_i$'s,  with relations 
$$e_i^2=a_i,\  e_ie_j=-e_je_i, \mathrm{if}\ i\neq j.$$
The Clifford group $C^*(q)$ is the group of homogeneous invertible elements $x$ of $Cl(q)$ such that $xvx^{-1}\in V$ for all $v\in V$. The algebra $Cl(q)$ is endowed with an involutory anti-automorphism $x\rightarrow x_t$ with $(x_1\cdots x_m)_t=(x_m\cdots x_1)$ for $x_i\in V$. The map $Cl(q)\rightarrow Cl(q)$ defined by $x\rightarrow x_tx$ restricts to a group homomorphism $sp: C^*(q)\rightarrow K^{\times}$. This is the {\it spinor norm} of $C^*(q)$. We define the group ${\bf Pin}(q)$ as the kernel of the spinor norm homomorphism. The orthogonal map $v\rightarrow-v$ on $(V, q)$ extends to an involutory automorphism $I$ of $Cl(q)$. We let $r: {\bf Pin}(q)\rightarrow {\bf O}(q)$ be the group homomorphism given by $r(x): v\rightarrow I(x)vx^{-1}$. Let $n$ be an integer, let $V=(K^s)^n$ be the direct sum of $n$ copies of $K^s$ and let $t$ be the unit form on $V$ with 
$$t(f_i)=1, t(f_i, f_j)=0, i\neq j, $$
where $\{f_i, 1\leq i\leq n\}$ is the canonical  basis of $V$. We set ${\bf O}_n(K^s)={\bf O}(t)$ (resp. $ {\bf Pin}_n(K^s)= {\bf Pin}(t)$). The homomorphism $r$ yields  an exact sequence of groups
\begin{equation}\label{pinm} 1\rightarrow {\bf Z}/2{\bf Z}\rightarrow  {\bf Pin}_n(K^s)\rightarrow {\bf O}_n(K^s)\rightarrow 1, \end{equation}
where ${\bf Z}/2{\bf Z}$ is the group with two elements. 

We let  $G$  be a group of  order $n$ and let  $f:G\rightarrow S_n$ be the group homomorphism induced by left multiplication of $G$  on itself. We  denote by $i$  the standard embedding $S_n\rightarrow \mathbf{O}_{n}(K^s)$.  Pulling back (\ref{pinm}) by $i\circ f$ provides us with an exact sequence 

\begin{equation}\label{pin}   1\rightarrow {\bf Z}/2{\bf Z} \rightarrow {\bf Pin}(G) \rightarrow G\rightarrow  1. \end{equation}
We observe that since the isomorphism  $S(G)\rightarrow S_n$ is defined up to conjugacy, the class of $H^2(G, {\bf Z}/2{\bf Z})$ attached to  the group extension (\ref {pin})  is well-defined. 

\subsection {Proof of Theorem \ref{main} and Corollaries  \ref{triv} and \ref{trivd}} 
\
\subsubsection {Proof of Theorem \ref{main} and Corollary \ref{triv}.} The proof of  Theorem \ref{main} is a consequence of the equality (\ref{S}) and the following proposition: 
\begin{prop}\label {lift2} Let $G$ be a group of even order $n$. Then the following properties are equivalent: 
\begin{enumerate}
\item the group extension  ${\bf Pin}(G)$ has the $2$-lift  property;
\item $n\equiv 0\ \mathrm{or}\ 2\ \mathrm{mod}\ 8$.
\end{enumerate}
\end{prop}

\begin{proof}   Take any
element $z$ of order two in $G$.  Then the orbits of the left multiplication
by $z$ on $S_n$ all have order two. So  $z':=f(z)$  is the product of $n/2$
disjoint transpositions in $S_n$.  For each transposition $(i,j)$ of $S_{n}$,
we can construct a lift  to the Clifford algebra of $t$ by taking $\varepsilon_{i,j}=(e_i - e_j)/\sqrt{2}$. One easily checks  that each of these belongs  to ${\bf Pin}_n(K^s)$ and  
 has  square $1$.  Moreover $\varepsilon_{i,j}.\varepsilon_{k,l}=- \varepsilon_{k,l}.\varepsilon_{i, j}$  whenever $(i, j)$ and $(k, l)$ are disjoint transpositions of $S_n$. So,  by counting how many sign changes occur as we move
lifts of transpositions past each other, we see that the square of a lift of $z'$ is the identity if and only if $\frac{n}{2}(\frac{n}{2}-1)\ \equiv 0\ \mathrm{mod}\ 4$. This proves the equivalence. 
\end{proof} 
  We now return to the proof of the theorem; we  let  $L/K$ be a $G$-Galois algebra of  degree  $n$,  $n\equiv 0\ \mathrm{or}\ 2\ \mathrm{mod}\ 8$ and we assume that $G$ is $2$-reduced.  By Proposition \ref{lift2} we know that ${\bf Pin}(G)$ has the $2$-lift property;   since $G$ is $2$-reduced,  this implies  that the group extension (\ref{pin}) is split and so   the class $c_G$ is trivial.  Therefore Theorem \ref{main}  follows from the equality (\ref{S}) whereas  Corollary \ref{triv} is a consequence of  Theorem \ref {main} and Proposition \ref {ww_1}. \hfill   $\square$
  \newline
 
  \noindent{\bf Remarks}  {\bf 1)} One  should note that,  in order to prove that $w_2(q_L)=0$,  \cite{Serre84}, the equality (\ref{S})  can be replaced by a  slightly weaker result (see \cite{Cassou}, Remark 6.6). 
  
  \noindent {\bf 2)}   Suppose that  $G$ is the group $PSL_2(\mathbb F_q), q\equiv 5\ \mathrm{mod}\ 8$. This is a group of order $n=q(q^2-1)/2$ with  elementary abelian Sylow $2$-subgroups. It follows from Theorem \ref{2red} that $G$ is $2$-reduced. However, since $n\equiv 4\ \mathrm{mod}\  8$, we deduce from Proposition \ref{lift2} that ${\bf Pin}(G)$ does not have the $2$-lift property and so that (\ref {pin}) is not split. It can be proved in this case that ${\bf Pin}(G)=SL_2(\mathbb F_q)$ whose Sylow $2$-subgroups are quaternion groups of order $8$. 
  \subsubsection {Proof of Corollary \ref{trivd}} 
  If  $L/K$ is a $G$-Galois algebra and $S$  a Sylow $2$-subgroup  of $G$, we know that there exists  a field  extension $K'/K$  of odd degree,  an  $S$-Galois algebra $M/K'$ and  an isomorphism of $G$-Galois algebras over $K'$
\begin{equation}\label{odd}L':=K'\otimes_{K}L\simeq \mathrm{Ind}^G_S(M) \end{equation}
(see \cite{Bayer94}, Proposition 2.11). We recall that if  $\Phi_M: G_{K'}\rightarrow S$ is the group homomorphism attached to $M/K'$, then the composition of $\Phi_M$ by the canonical injection $S\rightarrow G$ is a group homomorphism attached to $\mathrm{Ind}^G_S(M)$.  From (\ref {odd})  we deduce an isometry of quadratic forms 
$q_{L'}\simeq m\otimes q_M$ where $m$ is the index of $S$ in $G$. Since $S$ is a subgroup of $H$ we may consider the $H$-Galois algebra $E=\mathrm{Ind}^H_S(M)$.  As a $K'$-algebra $E$ is the product of $r$ copies of $M$ where $r$ is the index of $S$ in $H$. Hence  we obtain an isometry of quadratic forms $q_E\simeq r\otimes q_M$. Applying Theorem \ref{main}  to the $H$-Galois algebra $E$ we obtain that $w_1(q_E)=w_2(q_E)=0$. Since $r$ and $m$ are odd integers,  it suffices to apply (\ref{rel}) to deduce from the triviality of the Hasse-Witt invariants of $q_E$ in degree $1$ and $2$    that $w_1(q_M)=w_2(q_M)=0$ and so that $w_1(q_{L'})=w_2(q_{L'})=0$.  The group $G_{K'}$ is a subgroup of $G_K$ of odd index, therefore the restriction maps 
$$\mathrm{Res}^{G_K}_{G_{K'}}: H^i(G_K, {\bf Z}/2{\bf Z})\rightarrow H^i(G_{K'}, {\bf Z}/2{\bf Z})$$ are injective.  Since $\mathrm{Res}^{G_K}_{G_{K'}}w_i(q_L)=w_i(q_{L'})$ for each  integer $i$,  we conclude that $w_1(q_L)=w_2(q_L)=0$.\hfill $\square$
  \subsection{  Further results for Hasse-Witt invariants of the trace form }  
  Let  $L/K$ be    a $G$-Galois algebra. If   $G$ is the direct product of the subgroups $G_1$ and $G_2$   we set   $L_1:=L^{G_2}$  and  $L_2:=L^{G_1}$.  Then $L_1$ and $L_2$ are respectively $G_1$ and $G_2$-Galois algebras and $L$ and $L_1\otimes_K L_2$ are isomorphic $K$-algebras. This implies an  isometry of the $K$-forms 
\begin{equation}\label{isom} q_L\simeq q_{L_1}\otimes q_{L_2}. 
\end{equation}
For the sake of simplicity we set   
$$H(G_K, {\bf Z}/2{\bf Z})^\times=\{1+a_1+a_2\in \bigoplus_{0\leq i\leq 2}H^i(G_K, {\bf Z}/2{\bf Z}); a_i\in H^i(G_K, {\bf Z}/2{\bf Z})\}.$$
This is an abelian group under the law 
$$(1+a_1+a_2)(1+b_1+b_2)=(1+(a_1+b_1)+(a_2+b_2+ (a_1)(b_1)).$$
For a form $q$ we set $w(q):=1+w_1(q)+w_2(q) \in H(G_K, {\bf Z}/2{\bf Z})^\times$.  We recall that $w(q_1\oplus q_2)=w(q_1)w(q_2)$.
\begin{prop}\label{syl2}  Let $L/K$ be a $G$-Galois algebra and let $S$ be a Sylow $2$-subgroup of $G$. We assume that $S$ is the direct product of non-trivial subgroups $G_1$ and  $G_2$ and that either $G_1$ or $G_2$ is non-cyclic. Then one has the equalities:
$$w_1(q_L)=w_2(q_L)=0.$$
\end{prop}
\begin{proof} By using once again \cite{Bayer94} Proposition 2.1.1 it is easy to check that we may assume that $G=S$.  Suppose that $G_2$ is a non-cyclic group of order $n$.  By  (\ref {isom}) we have an isometry  of  quadratic forms 
$ q_L\simeq q_{L_1}\otimes q_{L_2}$. After choosing   a diagonalisation  $<a_1, \cdots, a_r>$ of $q_{L_1}$,   we obtain an isometry 
\begin{equation}
q_L\simeq \bigoplus_{1\leq i\leq r}<a_i>\otimes q_{L_2}.
 \end{equation}
 By \cite{Ber} Proposition 1.1 we know  that 
 \begin{equation}\label{bes} w(<a>\otimes q_{L_2})= 1+ n(a)+w_1(q_{L_2})+\binom {n}{2}(a)\cdot (a)+(n-1)(a)\cdot w_1(q_{L_2})+w_2(q_{L_2}) 
 \end{equation} for  any element $a\in K^{\times}$. Therefore, 
 since $n\ \equiv  0\  \mathrm{mod}\ 4$ and $G_2$ is non-cyclic,  it follows from (\ref {bes}) that   $w(<a_i>\otimes q_{L_2})=1+w_2(q_{L_2})$ for each integer $i$. Therefore $w(q_L)=(1+w_2(q_{L_2}))^r=1$ since $r$ is a power of $2$.
 \end{proof}
\begin{cor} Let L/K be a G-Galois algebra and let S be a Sylow 2-subgroup of G. We assume that $S$ is a non-metacyclic abelian group. Then one has the equalities
$$w_1(q_L)=w_2(q_L)=0.$$
\end{cor}
\begin {proof} Since $S$ is abelian it has a canonical decomposition into a product of cyclic groups. Since $S$ is non-metacyclic the decomposition of $S$ contains at least three factors. Therefore $S$ satisfies the hypotheses of Proposition \ref{syl2}.
\end{proof}
When the group $G$ is abelian it decomposes into a direct product $S\times S' $ where $S$ is the Sylow $2$-subgroup of $G$ and $S'$ is of odd order $m$ say. Since $S'$ is of odd order,  $q_{L^S}\simeq <1, \cdots, 1>$ by \cite{Bayer90} and so  $q_L$ is isomomorphic to $m\otimes q _E$ where $E$ is the $S$-Galois algebra $L^{S'}$. We assume that  $S$ is of order $2^r$, with $r\geq 3$,  (for $r\leq 2$ the form $q_L$ has been described in \cite{Bayer94} Section 6.1). If $S$ is either cyclic or equal to  a direct product of $s\geq 3$ non-trivial cyclic groups we have computed the Hasse-Witt invariants $w_1(q_L)$ and $w_2(q_L)$ (see Theorem \ref {main} and Proposition \ref{syl2}). We now assume that $S$ is product of two cyclic groups. We know that $w_1(q_L)=0$;  our aim is now to compute $w_2(q_L)$. In general we  observe that $S$ is not $2$-reduced in this case (see Section 3.3, Remarks 3)). We write  $S=S_1\times S_2$  where  $S_i$ is of order $2^{r_i}$ for $i\in \{1, 2\}$ and $r_1\geq 1, r_2\geq 2$. We set $E_1=E^{S_2}, E_2=E^{S_1}$ and we denote by $d_i$ the discriminant $d_{E_i/K}$. 
\begin{prop} Let $G$ be an abelian group  and let $L/K$ be a $G$-Galois algebra. We assume that the Sylow $2$-subgroup $S$ of $G$ is a product of two non-trivial cyclic groups. Then we have: 
\begin {enumerate}
\item $w_2(q_L)=(d_1d_2,d_2)$ if   $S$ has a direct factor of order $2$; 
\item $w_2(q_L)=(d_1, d_2)$ otherwise.  
\end{enumerate}  
\end{prop}
\begin{proof} Since $E$ is a $S$-Galois algebra and $S$ is non-cyclic we know that $w(q_E)=1+w_2(q_E)$.  Since $q_L$ is isometric to  $m\otimes q_E$,  then  $w(q_L)=w(q_E)^m=(1+w_2(q_E))^m$ and so, since $m$ is odd,  we conclude that 
$w_2(q_L)=w_2(q_E)$. From the isomorphism of algebras $E\simeq E_1\otimes_KE_2$ we deduce the isometry of forms $q_E\simeq q_{E_1}\otimes q_{E_2}$. If  $S$ has a direct factor of order $2$, then $q_{E_1}$ is of rank $2$ and $q_{E_2}$ is of rank $2^r, r\geq 2$. We choose a diagonalisation 
$<a_1, a_2>$ of $q_{E_1}$. Using (\ref {bes}), we obtain that 
\begin{equation}
 w(q_E)=\prod_{1\leq i\leq2}(1+d_2 +((a_i)\cdot d_2+w_2(q_{E_2})), 
 \end{equation}
  and therefore that $w(q_E)=1+(d_1d_2, d_2)$. We now suppose that $S_1$ is of order $2^s$ with $s\geq 2$. Then,  for $1\leq i\leq  2^{s-1}$, we  can choose elements $a_i$ and $b_i$ in $K^\times$ such that $$q_{E_1}=\bigoplus_{1\leq i\leq 2^{s-1}}<a_i, b_i>.$$ Therefore one has: 
  \begin{equation}\label{ciccic}w(q_E)=\prod_{1\leq i\leq2^{s-1}}w(<a_i, b_i>\otimes q_{E_2})=\prod_{1\leq i\leq 2^{s-1}}(1+(d_1(i)d_2, d_2))\end{equation}
  with $d_1(i)=a_ib_i$. It follows from (\ref {ciccic}) that 
  $$w(q_E)=1 +(2^{s-1}(d_2, d_2)+\sum_{1\leq i\leq 2^{s-1}} (d_1(i), d_2))=1+(d_1, d_2)$$
  as required.
\end{proof}

\section {Global fields }
In this section $K$ is either a global  field of characteristic different from $2$ or a number field. 

\subsection { Proof of Corollaries \ref {ff}, \ref{numb1}, \ref{im} and \ref {numb2}}  We first observe that Corollary \ref{im} is an immediate consequence of Corollary \ref{numb1}.  We  let  $G$ be a  group of order $n$; we denote by $S$ a Sylow $2$-subgroup of $G$.  We consider a $G$-Galois algebra $L/K$ of degree $n$.  For a place $v$ of $K$ and a quadratic form $r$ over $K$ we let $r_v$ be the extended form $K_v\otimes_Kr$.  For any   place $v$ of $K$ we know that   $w_i(q_{L, v})$ is the image of $w_i(q_L)$ by the restriction map induced by  the injection $G_{K_v}\rightarrow G_K$. 

We first assume that the group $S$ is non-cyclic. Let  us  denote by $t$ the unit  form $X_1^2+...+X_n^2$ over $K$. For each    place $v$ of $K$    it follows from  Corollary \ref{triv} that 
 $$w_i(q_{L,v})=w_i(t_v)=0,\  i \in \{1, 2\} $$
 so   that    $q_{L, v}$ and $ t_v$ are isometric as forms over the local field $K_v$ for any non-archimedean place.  Since any place of a global function field is non-archimedean,  using   Hasse-Minkowski Theorem,  we conclude that  the trace form $q_L$ is isometric to $t$ and  Corollary \ref{ff} i) is proved. 
  
  Suppose now that $K$ is a number field. Let   $v$ be  an archimedean  place. If $v$ is  complex,   the forms $q_{L, v}$ and $ t_v$ are isometric over ${\bf C}$ because  they have the same rank. We now assume that $v$ is real.  If $\sigma(L_v)$ is trivial then $L_v/K_v$ is completely split and so $q_{L, v}\simeq t_v$. If  $\sigma(L_v)$  is non-trivial, then $L_v$ is isomorphic as a $K_v$-algebra to a product of 
 $n/2$ copies of ${\bf C}$.  The trace form of ${\bf  C}/{\bf  R}$ is isometric to $<1, -1>$ and thus   
 $q_{L,v}$ is isometric to $n/2$ copies of $<1, -1>$.  Since  $q_L$ is isometric to  $t$ if and only if $q_{L, v}\simeq t_v$ for any place $v$ of $K$, then we conclude that $L/K$ has a self-dual basis  if and only if $\sigma(L_v)=1$ for any real place. This proves Corollary \ref{numb1}. 
 
 For  $K={\bf Q}$ there exists a unique non-archimedean  place $v_\infty$. If  $L/{\bf  Q}$ is totally real then   $\sigma (L_{v_{\infty}})=1$ and  so  it follows from Corollary \ref{numb1}  that $q_L\simeq <1, \cdots, 1>$. Suppose now that $L/{\bf Q}$ is totally imaginary. We denote by $r$ the ${\bf  Q}$-quadratic form $(n/2)\otimes <1, -1>$.  Since 
$ n\equiv 0\  \mathrm{mod}\ 8$,  using  (\ref{rel}), we check  that $w_1(r)=w_2(r)=0$, and  therefore, using  Corollary \ref{triv}, we deduce that $w_i(q_L)=w_i(r)$ for 
$i \in \{1, 2\}$. Moreover since $\sigma (L_{v_{\infty}})\neq 1$, then $q_{L, v_\infty}$ is isometric to $(n/2)\otimes <1, -1>$ as ${\bf R}$-forms. We conclude that $q_L$ and $r$ having the same Hasse-Witt invariants in degree $1$ and $2$ and  having the same signature are isometric. Hence  Corollary \ref{numb2} (1) and (2) are proved.  

We now assume that the group $S$ is cyclic.  When $K$ is a global function field or is equal to ${\bf  Q}$, we let $s$ be the quadratic form $<2, 2d_{L/K}, 1, \cdots, 1>$. One easily checks  that $w_i(q_L)=w_i(s)$ for $iÊ\in \{1, 2\}$.  If   $K={\bf Q}$ and $L$ is totally real, then  the forms $q_L$ and $s$ have the same signature. We conclude that $q\simeq s$ when $K$ is either a function field or when $L/{\bf Q}$ is totally real. This completes the proof of Corollary \ref{ff} and proves Corollary \ref{numb2} iii). Setting $s'=(\frac{n}{2}-1)\otimes<1, -1>\oplus <(-1)^{(\frac{n}{2}-1)}2, 2d_L>$, we complete the proof of Corollary \ref {numb2} by hand checking the equalities of the signatures and the Hasse-Witt invariants in degree $1$ and $2$ of the forms $q_L$ and $s'$. \hfill $\square$
\subsection { Proof of Proposition \ref{cp}. } We use the notation of Section 2.1. By a local field we mean a field,  complete with respect to a fixed discrete valuation,  that has a perfect residue field  of  positive characteristic .
 \begin{lem}\label{floc} Let $K$ be  a local field with  residual characteristic different from $2$ and  let $G$ be a finite group with non-metacyclic Sylow $2$-subgroups. Then the trace form of any $G$-Galois algebra over $K$ is isometric to the unit form.
\end{lem}
\begin{proof} Let $L/K$ be a $G$-Galois algebra,  $\chi \in \mathrm{Hom}^{alg}(L, K^s)$  and  $\Phi_L: G_K\rightarrow G$ be the morphism attached to $L$.  We set 
$H=\mathrm{Im}(\Phi_L)$. Since $G$ is non-cyclic we know from Proposition \ref{ww_1} that $w_1(q_L)=0$. Moreover,  it follows from (\ref{rel}) that 
$w_2(q_L)=\binom{m}{2}w_1(q_E)\cdot w_1(q_E) +mw_2(q_E)$,  where  $E$ denotes the subfield $\chi(L)$ of $K^s$ and $m$ is the index of $H$ in $G$. Let $S$ be the Sylow $2$-subgroup of $H$. Since the residual characteristic of $K$ is different from $2$,  the extension $E/E^S$ is at most tamely ramified and so $S$ is metacyclic (see \cite{Serrecl}, Chapter IV). Let $S'$ be a Sylow $2$-subgroup of $G$ containing $S$ and let $2^r$ be the index of $S$ in $S'$. The integer $2^r$ divides $m$ and $r \geq 1$ since $S'$ is not metacyclic. If $S$ is not cyclic it follows from Proposition \ref{ww_1} that $w_1(q_E)=0$ and so that  $w_2(q_L)=0$ since $m$ is even by hypothesis.  If now $S$ is cyclic, since $S'$ is not metacyclic,  then necessarily  $r\geq 2$ and so $\binom{m}{2}$ is even and once again $w_2(q_L)=0$.  We conclude that,  if $n$ denotes the degree of $L/K$,  the form $q_L$ and the unit form of rank $n$ having the same Hasse-Witt invariants in degree $1$ and $2$ are isometric.
 \end{proof}
 Suppose  now that $L$ is a  $G$-Galois algebra over $K$ with non-metacyclic Sylow $2$-subgroups. If $K$ is a global function field of characteristic different from $2$, following  the proof of Corollary \ref{ff}, we deduce from Lemma \ref{floc} that $q_L$ and the unit form $t$ are locally isometric at every place $v$ of $K$ and so we conclude that they are globally isometric. Similarly, when  $K={\bf  Q}$, we deduce that $q_L$ and the unit form are locally isometric at every place $v \neq 2$.  Using  Hasse reciprocity law we conclude that the same is true at $v=2$ and therefore that $q_L$ and $t$ are isometric. 
 \hfill $\square$

 \section {Trace form of Galois covers of  a scheme}
 Our goal is to use the results of the previous sections on group extensions and group cohomology in a geometric set-up, namely  when   we replace the  base field $K$ by a  connected scheme $Y$  in which $2$ is invertible and  the Galois $G$-algebra $L/K$ by a Galois $G$-cover $X\rightarrow Y$. This can be done thanks to  the  generalisation of Serre's comparison formula for \'etale covers of schemes obtained  by Kahn, Esnault and Viehweg in \cite{EKV},  Theorem 2.3.
 
  We fix a connected scheme $Y$ in which $2$ is invertible. We recall that a symmetric bundle over $Y$ is given by $(V, q)$ where $V$ is a locally free $\mathcal{O}_Y$-module and 
 $$q: V\otimes _{\mathcal{O}_Y}V\rightarrow \mathcal{O}_Y$$
is a symmetric morphism of $\mathcal{O}_Y$-modules. Let $V^{\vee}$ be the dual of $V$. The form $q$ induces a  morphism 
$\varphi_q: V\rightarrow V^{\vee}$ of $\mathcal{O}_Y$-modules; we assume that $\varphi_q$ is an isomorphism. In this section we consider  symmetric bundles attached to finite \'etale covers of $Y$. More precisely if $\pi: X\rightarrow Y$ is a finite \'etale cover  we denote by $(V_X, q_X)$ the symmetric bundle where $V_X=\pi_*(\mathcal{O}_X)$ and 
$$q_X: V_X\otimes_{\mathcal{O}_Y}V_X\rightarrow \mathcal{O}_Y$$  is defined over any affine open subcheme $\mathrm{Spec}(A)\subseteq Y$ by 
$$(x, y)\rightarrow \mathrm{Tr}_{B/A}(xy), \forall\  x, y \in B$$ 
where $\mathrm{Spec}(B)=\pi^{-1}(\mathrm{Spec}(A))$. For any symmetric bundle $(V, q)$ and any integer $m\geq 1$ one can define the \emph{$m$-th  Hasse-Witt invariant} of $q$ as an element of the \'etale cohomology group $H_{et}^m(Y, {\bf Z}/2{\bf Z})$ (see \cite{EKV} Section 1 or  \cite{Cassou} Section 4.5); indeed when $Y=\mathrm{Spec}(K)$   and $X=\mathrm{Spec}(L)$,  where $L/K$ is a finite separable algebra, then $q_X$ is defined by  the trace form $q_L$ of $L/K$ and the Hasse-Witt invariants of $q_X$ coincide with the Hasse-Witt invariants of $q_L$  introduced in Section 2.2. 

Let $\pi_1(Y,\overline{y})$ be the fundamental group of $Y$ based at some geometric point $\overline{y}$. We consider a finite group $G$ and a finite \'etale Galois cover $\pi: X\rightarrow Y$ of group $G=\mathrm{Aut}_Y(X)$.
Hence the finite set $\mathrm{Hom}_Y(\overline{y}, X)$ is endowed on the one hand with a simply transitive action of $G$,  induced by the action of $G$  on $X$,  and on the other hand with a continuous action of $\pi_1(Y,\overline{y})$. Following the lines of Section 2.1,  the choice of a point $\chi\in \mathrm{Hom}_Y(\overline{y}, X)$ gives a surjective group homomorphism $\Phi_X: \pi_1 (Y,\overline{y})\rightarrow G$, which does not depend on $\chi$ up to conjugacy. By composition with $f: G\rightarrow S_n$, we obtain a group homomorphism $\pi_1(Y,\overline{y})\rightarrow S_n$. Let  $K^s$ be  a separable closure of the residue field of some point of $Y$.  We  obtain an orthogonal representation 
$$\rho_X: \pi_1(Y,\overline{y})\rightarrow G\rightarrow S_n\rightarrow  {\bf O}_n(K^s)$$
by composing $f\circ \Phi_X$  with the standard embedding $i: S_n\rightarrow {\bf O}_n(K^s)$. We can now associate cohomological invariants to the orthogonal representation $\rho_X$.  The first class $w_1(\rho_X)$ is the group homomorphism $\mathrm{det}\circ \rho \in H^1(\pi_1(Y,\overline{y}), {\bf Z}/2{\bf Z})$.  
The second class $w_2(\rho_X)$ is defined as the pull-back by $\rho_X$ of the group extension (\ref {pinm}), Section 4.2. It follows from the definition of $\rho_X$ that 
$w_2(\rho_X)= \Phi_X^*(c_G)$ where $c_G\in H^2(G, {\bf Z}/2{\bf Z})$ is defined by the group extension 
 $$1\rightarrow {\bf Z}/2{\bf Z} \rightarrow {\bf Pin}(G) \rightarrow G\rightarrow  1$$
 introduced in (\ref{pin}), Section 4.2. Finally we define $w_i(\pi)\in H_{et}^i(Y, {\bf Z}/2{\bf Z})$, $i\in \{1, 2\}$,  as the image of $w_i(\rho_X)$ by the canonical group homomorphism $can: H^i(\pi_1(Y,\overline{y}), {\bf Z}/2{\bf Z})\rightarrow H_{et}^i(Y, {\bf Z}/2{\bf Z})$. We note that $can$ is an isomorphism for $i=1$ and an injective morphism for $i=2$.  Moreover $w_i(\pi)$ does not depend of the choice of the geometric point $\overline{y}$.

 For any unit $a\in \Gamma (Y, {\bf G}_m)$ we denote by $(a)\in H_{et}^1(Y, {\bf Z}/2{\bf Z})$ the  image of $a$ by the boundary map associated to the Kummer exact sequence of etales sheaves 
\[\xymatrix{0\ar[r]^{}&
{\bf Z}/2{\bf Z}\ar[r]^{}&{\bf G}_m\ar[r]^{2}& {\bf G}_m\ar[r]^{} & 0.\\
}\]
Theorem \ref{main} and Corollary \ref{triv} can be generalised as follows:
\begin{thm} Let $G$ be a $2$-reduced  group of order $n$, $n\equiv  0\ \mathrm{or}\ 2\ \mathrm{mod}\ 8$.  Then for any $G$-Galois cover  $\pi: X\rightarrow Y$ over $Y$  one has: 
$$w_2(q_X)=(2)\cdot w_1(\pi).$$
Moreover if the Sylow $2$-subgroups of $G$ are non-cyclic. Then 
$$w_1(q_X)=w_2(q_X)=0.$$

\end{thm}
\begin{proof} We consider the orthogonal representation $\rho_X: \pi_1(Y,\overline{y})\rightarrow {\bf O}_n(K^s)$ attached to  $\pi: X\rightarrow Y$. Since the group $G$ is $2$-reduced it follows from Proposition \ref{lift2} that the class $c_G$ is trivial and so that $w_2(\rho_X)=\Phi_X^*(c_G)=0$. Moreover if the Sylow $2$-subgroups of $G$ are non-cyclic we know from Lemma \ref{et}  that $\mathrm{Im}(f)$ is contained in $A_n$ and therefore that $w_1(\rho_X)=0$.    We deduce from  \cite{EKV}  Theorem 2.3  the following equalities:  
\begin{equation}\label{K}
w_1(q_X)=w_1(\pi)\ \mathrm{and}\  \ w_2(q_X)=w_2(\pi)+(2)\cdot w_1(\pi). 
\end{equation}
Therefore the theorem follows immediately from  (\ref{K}) and the equalities $w_i(\pi)=can(w_i(\rho_X))$. 

\end{proof}

\end{document}